\RequirePackage{fix-cm}
\documentclass[smallextended]{svjour3}       
\smartqed  
\journalname{Calculus of Variations and Partial Differential Equations}
\usepackage{amsmath,bm,color,xcolor}
\usepackage[pdftex]{graphicx}
\usepackage{amsfonts}
\usepackage{enumerate}
\usepackage{amssymb}
\usepackage{mathrsfs}
\usepackage{cancel}
\usepackage{graphicx}
\usepackage{color}
\usepackage{url}

\newcommand{\om}{\Omega}

\newcommand{\R}{\mathbb{R}}

\newcommand{\ep}{\varepsilon}

\newcommand{\dom}{{\rm dom}\,}

\newcommand{\be}{\begin{eqnarray}}
\newcommand{\ee}{\end{eqnarray}}

\renewcommand{\leq}{\leqslant}
\renewcommand{\geq}{\geqslant}

\newcommand{\supp}{{\rm supp}\,}

\newcommand{\cof}{{\rm cof}}

\newcommand{\1}{{\mathbf 1}}

\newcommand{\av}{-\hspace{-.15in}\int}

\begin{document}

\title{Extension of convex functions from a hyperplane to a half-space}

\author{John M. Ball    \and  Christopher L. Horner    
}


\institute{
             Department of Mathematics, Heriot-Watt University and Maxwell Institute for Mathematical Sciences, Edinburgh. \email{jb101@hw.ac.uk}}
             

\date{}




\maketitle

\begin{abstract} It is shown that a possibly infinite-valued proper lower semicontinuous convex function on $\R^n$ has an extension to a convex function on the half-space $\R^n\times[0,\infty)$ which is finite and smooth on the open half-space $\R^n\times(0,\infty)$.  The result is applied to nonlinear elasticity, where it clarifies  how the condition of polyconvexity of the free-energy density $\psi(Dy)$ is best expressed when $\psi(A)\to\infty$ as $\det A\to 0+$. 
\end{abstract}
\keywords{convex extension, infimal convolution, polyconvex}
\section{Introduction}
\label{intro}
The main purpose of this paper is to prove the following theorem, giving an extension of a possibly infinite-valued proper lower semicontinuous convex function on $\R^n$ to a convex function on the half-space $\R^n\times [0,\infty)$ which is finite and smooth on the open half-space $\R^n\times (0,\infty)$. \\

\begin{theorem}
\label{lifting}
Let $\Phi:\R^n\to (-\infty,\infty]$ be a proper lower semicontinuous convex function. Then there exists a lower semicontinuous convex function $$\varphi:[0,\infty)\times\R^n \to (-\infty,\infty],\; \varphi=\varphi(x,y),$$such that

$(i)$ $\varphi(0,y)=\Phi(y)$ for all $y\in \R^n$,

$(ii)$  $\lim_{x\to 0+}\varphi(x,y)= \Phi(y)$ for each $y\in \R^n$.

$(iii)$ $\varphi:(0,\infty)\times\R^n\to \R$ is smooth,\\
If $\Phi\geq 0$, then $\varphi$ can be chosen so that $\varphi\geq 0$, and if $\Phi:\R^n\to (-\infty,\infty]$ is continuous, $\varphi:[0,\infty)\times\R^n \to (-\infty,\infty]$ can be chosen to be continuous. If $\Phi$ is strictly convex on ${\rm dom}\,\Phi:=\{y\in\R^n:\Phi(y)<\infty\}$ then $\varphi$ can be chosen to be strictly convex on $(0,\infty)\times\R^n$.
\end{theorem}
The following result  is an immediate consequence (setting $\Phi^{(j)}(y)=\varphi(j^{-1},y)$).\\
\begin{corollary}\label{cor}
Let $\Phi:\R^n\to (-\infty,\infty]$ be a proper lower semicontinuous convex function. Then there exists a sequence $\Phi^{(j)}$ of smooth convex functions on $\R^n$ such that $\lim_{j\to \infty}\Phi^{(j)}(y)= \Phi(y)$ for each $y\in \R^n$.\\
\end{corollary}

The theorem applies, for example, to the case  when $\Phi$ is the indicator function $i_K$ of a nonempty closed convex subset $K\subset\R^n$, defined by
$$i_K(y)=\left\{\begin{array}{cc}0&\text{if }x\in K\\
\infty&\text{if }x\not\in K.\end{array}\right.$$
With $K=\{0\}$ a suitable smooth strictly convex extension is then given by $\varphi(x,y)=\theta(x,y) -\frac{x}{x+1}$, where
 \be
 \label{theta}
 \theta(x,y)=\left\{\begin{array}{cc}
 \frac{|y|^2}{x},&\text{if }x>0, y\in\R^n,\\
 0,&\text{if }x=0, y=0,\\
 +\infty,& \text{otherwise},
 \end{array}
 \right.
\ee
which follows as a special case of \eqref{formula} (or  \eqref{infconv}) below. (We note that with $y$ momentum and $x$ density the convexity of $\theta$ plays an important role in optimal transport, as noted in \cite{benamoubrenier}.)

The theorem was motivated by the problem of proving the existence of energy minimizers in 3D nonlinear elasticity under the assumption of  polyconvexity of the free-energy density. In   \cite{j26} an apparently weaker version of the polyconvexity condition given in \cite{j8} was used. That this version is indeed weaker follows from Theorem \ref{lifting}, and   this is explained in Section \ref{polyconvexity}.

\section {Proof of Theorem \ref{lifting}}
We first show the existence of an extension $\tilde\varphi$ satisfying (i), (ii), which in addition is (strictly) decreasing in $x$, giving two different proofs. The first proof is the more direct and provides a wide range of possible extensions, while the second uses infimal convolution and is convenient for proving the assertion in the theorem regarding strict convexity.\\

\begin{proposition}
\label{finiteextension}
Under the assumptions of Theorem \ref{lifting} there exists a lower semicontinuous convex extension $\tilde\varphi=\tilde\varphi(x,y)$ of $\Phi$ to $[0,\infty)\times\R^n $ that is finite for $x>0$, decreasing in $x$, and such that $\lim_{x\to 0+}\tilde\varphi(x,y)=\Phi(y)$ for each $y\in\R^n$. \\
\end{proposition}
\noindent{\it 1st Proof.} 
We first note that $\Phi$ is the supremum of a family of affine functions:
\be
\label{2} \Phi(y)=\sup_{(\alpha,b)\in S}(\alpha+b\cdot y),\;\text{for all }y\in\R^n,
\ee
for some nonempty set $S\subset\R^{n+1}$. This is a standard result; see, for example, \cite[Proposition 3.1]{ekelandtemam}, \cite [Theorem 12.1]{rockafellar70}. (In Remark \ref{remark} below we note that we can take the family of affine functions to consist of exact affine minorants, but this is not needed for the proof.)

Let $\psi:[0,\infty)\to [0,\infty)$ satisfy 
\be
\label{3}
\lim_{t\to\infty}\frac{\psi(t)}{t}=\infty.
\ee
We claim that 
\be
\label{formula}
\tilde\varphi(x,y):=\sup_{(\alpha,b)\in S}\left(\alpha+b\cdot y-\psi\left(|\alpha|+|b|\right)x\right)
\ee
provides a suitable convex extension. Indeed by \eqref{2} $\tilde\varphi(0,y)=\Phi(y)$ for all $y\in\R^n$, and since it is the supremum of continuous affine functions  $\tilde\varphi$ is convex and lower semicontinuous. 

Given $x>0, y\in \R^n$, by \eqref{3} there exists $M(x,y)>0$ such that 
\be
\label{4}
\frac{\psi(|\alpha|+|b|)}{|\alpha|+|b|}>x^{-1} \max(1,|y|)\;\; \text{    if }|\alpha|+|b|> M(x,y).
\ee
Hence for $|\alpha|+|b|> M(x,y)$ we have
\be
\label{5}
\alpha+b\cdot y-\psi(|\alpha|+|b|)x \leq 0.
\ee
Therefore, since $\psi\geq 0$, $\tilde\varphi(x,y)\leq \max(1,|y|)M(x,y)<\infty$, as required. 

$\tilde\varphi(x,y)$ is nonincreasing in $x$, and can be made decreasing by adding $-x$ to $\tilde \varphi$.

     Since $\tilde\varphi$ is lower semicontinuous 
\be
\label{5a}
\Phi(y)=\tilde\varphi(0,y)\leq \liminf_{x\to 0+}\tilde\varphi(x,y)\leq\limsup_{x\to 0+}\tilde\varphi(x,y)&&\\&&\hspace{-1.1in}\leq\sup_{(\alpha,b)\in S}(\alpha+b\cdot y)=\Phi(y),\nonumber
\ee
so that 
\be
\label{5b}
\lim_{x\to 0+}\tilde\varphi(x,y)=\Phi(y).
\ee
 \\
 
 \noindent{\it 2nd Proof.}
 Define  $\theta:\R^{n+1}\to [0,\infty]$ by \eqref{theta}.
Note that $\theta$ is convex and lower semicontinuous on $\R^{n+1}$;  the convexity  follows, for example, from the identity
\be
\label{thetaconv}
\lambda\theta(x_1,y_1)+(1-\lambda)\theta(x_2,y_2)-\theta(\lambda (x_1,y_1) +(1-\lambda)(x_2,y_2))&\\
\nonumber &\hspace{-1.8in}=\frac{\lambda(1-\lambda)}{\lambda x_1+(1-\lambda)x_2}\left|\sqrt{\frac{x_1}{x_2}}y_2-\sqrt{\frac{x_2}{x_1}}y_1\right|^2\geq 0
\ee
for $\lambda\in[0,1]$ and $(x_1,y_1),(x_2,y_2)\in (0,\infty)\times\R^n$, and examining the behaviour of $\theta$ along lines in $\R^{n+1}$.\\
Let $\tilde\varphi=\Phi\,\square\,\theta$ be the infimal convolution of $\Phi$ and $\theta$ with respect to $y\in\R^n$ defined by
\be
\label{infconv}(\Phi\,\square\,\theta)(x,y)=\inf_{y'\in\R^{n}}\left(\Phi(y')+\theta(x,y-y')\right).
\ee
The convexity of $\Phi$ and $\theta$ implies that the function 
\be
\label{5ba}
h(x,y,y'):=\Phi(y')+\theta(x,y-y')
\ee
is convex on $\R^{2n+1}$. Hence by \cite[Prop. 8.26]{bauschkeetal} $\tilde\varphi(x,y)=\inf_{y'}h(x,y,y')$ is convex in $(x,y)$. Since $\Phi$ is proper, there exists $\bar y\in\R^n$ with $\Phi(\bar y)<\infty$. Therefore for $x>0$ we have that $\tilde\varphi(x,y)\leq \Phi(\bar y)+\frac{|y-\bar y|^2}{x}<\infty$. Also $\tilde\varphi(0,y)=\min\,(\Phi(y),\infty)=\Phi(y)$, so that $\tilde\varphi$ is an extension of $\Phi$. Furthermore, 
\be
\label{5baa}
\Phi(y)\geq \alpha +b\cdot y \text{ for all  }y\in\R^n\text{ and some }\alpha\in\R, b\in\R^n. 
\ee
Hence for $x>0$
\be
\label{5bb}
\tilde\varphi(x,y)\geq \inf_{y'\in\R^n}\left(\alpha+b\cdot y'+\frac{|y-y'|^2}{x}\right)=\alpha +b\cdot y-\frac{|b|^2x}{4}>-\infty,
\ee
so that $\tilde\varphi(x,y)$ is finite, and thus by convexity continuous on $(0,\infty)\times\R^n$.  If $\tilde\varphi$ were not lower semicontinuous there would exist a sequence  $(x_j,y_j)\to(0,y)$ and $y'_j$ with
\be
\label{5bc}
\sup_j\left(\Phi(y_j')+\frac{|y_j-y_j'|^2}{x_j}\right)<\Phi(y).
\ee
In particular the left-hand side of \eqref{5bc} is bounded, and so, using \eqref{5baa}, $y_j'\to y$. Thus by the lower semicontinuity of $\Phi$ the left-hand side is greater than or equal to $ \Phi(y)$, a contradiction. If $x_j\to 0+$ and $y\in\R^n$ then by the lower semicontinuity  $\Phi(y)\leq \liminf_{j\to\infty}\tilde\varphi(x_j,y)\leq \limsup_{j\to\infty}\tilde\varphi(x_j,y)\leq \Phi(y)$, so that $\lim_{x\to 0+}\tilde\varphi(x,y)=\Phi(y)$ as required.

Clearly $\tilde\varphi(x,y)$ is  nonincreasing in $x$. As defined it may not be decreasing (consider the case $\Phi\equiv 0$), but $\tilde\varphi(x,y)-x$ is decreasing in $x$ and satisfies the other requirements.
\qed
\begin{corollary}
\label{strictcor}
Assume in addition to the hypotheses of Proposition \ref{finiteextension} that $\Phi$ is strictly convex on ${\rm dom}\,\Phi$. Then $\tilde\varphi$ can be chosen so that in addition it is strictly convex on $(0,\infty)\times\R^n$.
\end{corollary}

\begin{proof}
We use the construction in the second proof of Proposition \ref{finiteextension}.  Fix $x>0$. Given $y\in\R^n$, by \eqref{5baa} the minimum of $h(y,z):=\Phi(z)+\frac{|y-z|^2}{x}$ for $z\in\R^n$ is attained by some $z=y'\in\dom\Phi$, and the strict convexity of $\Phi$ on $\dom\Phi$ implies that $y'$ is unique. Given distinct $y,\bar y\in\R^n$ let the corresponding unique minimizers be $y',\bar y'$ respectively.  For $\lambda\in(0,1)$ the strict convexity of $h$ on $\R^n\times\dom\Phi$ implies that
\begin{align}
\nonumber
\tilde\varphi(x,\lambda y+(1-\lambda)\bar y)&\leq
h(\lambda y+(1-\lambda)\bar y,\lambda y'+(1-\lambda)\bar y')\\&<\lambda h(y,y')+(1-\lambda)h(\bar y,\bar y')\\
&=\lambda\tilde\varphi(x,y)+(1-\lambda)\tilde\varphi(x,\bar y).\label{5bd}
\end{align}
Hence $\tilde\varphi (x,y)$ is  strictly convex in $y$. 

To complete the proof we use the following lemma.
\begin{lemma}
\label{strictc}
Let $f:(0,\infty)\times\R^n\to\R$ be convex with $f(x,y)$ strictly convex in $y$ for each $x$. If $\psi:(0,\infty)\to \R$ is strictly convex then $g(x,y):=f(x,y)+\psi(x)$ is strictly convex in $(x,y)$.
\end{lemma}
\begin{proof}
$g$ is convex. If $g$ were not strictly convex then there would exist distinct pairs $(x_1,y_1), (x_2,y_2)$ and $\lambda\in(0,1)$ with
\begin{align}
\label{5be}
f(\lambda x_1+(1-\lambda)x_2),\lambda y_1+(1-\lambda)y_2)+\psi(\lambda x_1+(1-\lambda)x_2)&=\\& \hspace{-2.5in}
\lambda f(x_1,y_1)+(1-\lambda)f(x_2,y_2) +\lambda \psi(x_1)+(1-\lambda)\psi(x_2).\nonumber
\end{align}
It follows from \eqref{5be} and the convexity of $\psi$ that $\psi(\lambda x_1+(1-\lambda)x_2)=\lambda \psi(x_1)+(1-\lambda)\psi(x_2)$, and since $\psi$ is strictly convex we must have $x_1=x_2$. But then \eqref{5be} contradicts the strict convexity of $f(x,y)$ in $y$.\qed
\end{proof}
\noindent Now let $\psi:[0,\infty)\to\R$ be strictly convex and decreasing with $\psi(0)=0$ (for example, $\psi(x)=-\frac{x}{x+1}$). Then, by Lemma \ref{strictc}, $\tilde\varphi(x,y)+\psi(x)$ is a suitable strictly convex extension.\qed
\end{proof}
To complete the proof of Theorem \ref{lifting} we mollify $\tilde\varphi$ as constructed in Proposition  \ref{finiteextension} with an $x$-dependent mollifier. Let $\rho=\rho(x,y)\geq 0$, $\rho\in C_0^\infty(\R^{n+1})$, $\supp \rho\subset (0,1)\times \R^n$, $\displaystyle\int_{\R^{n+1}}\rho\,dx\,dy=1$, and define for $(x,y)\in [0,\infty)\times \R^n$
\be
\label{6}
 \varphi(x,y)=\int_{\R^n}\int_0^1\rho(x',y')\tilde\varphi(x(1-x'),y-xy')\,dx'dy'.
\ee
 The integral is well defined since $\tilde\varphi$ is convex on $[0,\infty)\times\R^n$  and thus bounded below by a linear function, and the convexity of $\tilde\varphi$ also implies that   $\varphi$ is convex. 
Since $\tilde\varphi$ is lower semicontinuous, by Fatou's Lemma (valid because $\tilde\varphi$ is bounded below by a linear function) $\varphi$ is lower semicontinuous.   Furthermore
\be
\label{8}
\varphi(0,y)=\Phi(y).
\ee
 Making the change of variables $u=x(1-x'), v=y-xy'$ we have that for $x>0$
\be
\label{7}
 \varphi(x,y)=x^{-(n+1)}\int_{\R^n}\int_0^\infty\rho\left(\frac{x-u}{x},\frac{y-v}{x}\right)\tilde\varphi(u,v)\,du\,dv,
\ee
from which it follows that $\varphi$ is smooth for $x>0$.

   We next note that for any $y\in\R^n$, the convexity of $\varphi$ implies that
\be
\label{9a}
\varphi(x,y)\leq (1-x)\varphi(0,y)+x\varphi(1,y),
\ee
so that by \eqref{8}
\be
\label{9}
\limsup_{x\to 0+}\varphi(x,y)\leq\Phi(y).
\ee
But also, since $\varphi$ is lower semicontinuous,
\be
\label{11}
\Phi(y)\leq\liminf_{x\to 0+}\varphi(x,y).
\ee
Combining \eqref{9}, \eqref{11} we see that
\be
\label{12}
\lim_{x\to 0+}\varphi(x,y)=\Phi(y).
\ee
If $\Phi\geq 0$, then we can replace $\tilde\varphi$ by  $\max (\tilde\varphi, 0)$, so that $\varphi\geq 0$ also.

Suppose that $\Phi:\R^n\to (-\infty,\infty]$ is continuous, and let $x^{(j)}\to 0+$, $y^{(j)}\to y$ in $\R^n$. If $\Phi(y)=\infty$ then the lower semicontinuity of $\varphi$ implies that $\varphi(x^{(j)},y^{(j)})\to \varphi(0,y)=\infty$. If $\Phi(y)<\infty$ then the continuity of $\Phi$ implies that $\Phi(z)<\infty$ for $|z-y|$ sufficiently small. By what we have proved the sequence $\Phi^{(j)}(z):=\varphi(x^{(j)},z)$ of convex functions  converges pointwise to $\Phi$, and hence by \cite[Theorem 10.8]{rockafellar70} the convergence is uniform on a neighbourhood of $y$, so that again $\varphi(x^{(j)},y^{(j)})\to \varphi(0,y)=\Phi(y)$. Hence $\varphi:[0,\infty)\times\R^n\to (-\infty,\infty]$ is continuous.

Finally, if $\Phi$ is strictly convex on $\dom\Phi$ then by Corollary \ref{strictcor} we can suppose that $\tilde\varphi$ is strictly convex on $(0,\infty)\times\R^n$, so that $\varphi$ is strictly convex on $(0,\infty)\times\R^n$ by \eqref{6}. \qed

\begin{remark}\rm
\label{remark}
 In \eqref{2} we can take $S$ to consist of all points $(\Phi(y_0)-b(y_0)\cdot y_0, b(y_0))$ where $y_0$ belongs to the domain ${\rm dom}\,\partial\Phi$ of the subdifferential $\partial\Phi$ of $\Phi$ and $b(y_0)\in \partial\Phi(y_0)$. That is $\Phi$ is the supremum of all its exact affine minorants. This fact is not typically given in standard texts on convex analysis, although \cite[Corollary 3.21]{phelps} gives such a result for points $y$ where $\Phi(y)<\infty$. The result is stated (for Hilbert spaces) in the paper of Moreau \cite[Section 8.c]{moreau65} (see also \cite[Section 13]{moreau67}), and follows from his theorem \cite[Section 8.b]{moreau65}  (see also \cite[Theorem 24.9]{rockafellar70}) that if $\Phi, \Psi$ are proper lower semicontinuous convex functions with $\partial \Phi(y)\subset \partial \Psi(y)$ for all $y\in\R^n$ then $\Phi=\Psi+c$ for some constant $c$. Indeed if we define
\be
\label{13}
\Psi(y) =\sup_{y_0\in {\rm dom}\,\partial\Phi, b(y_0)\in\partial\Phi(y_0)}\Phi(y_0)+b(y_0)\cdot(y-y_0),
\ee
then $\Phi\geq\Psi$ and for any $y_0\in {\rm dom}\,\partial\Phi$ and $b(y_0)\in\partial\Phi(y_0)$ we have for all $y\in\R^n$
\be
\label{14}
\Psi(y)\geq \Phi(y_0)+b(y_0)\cdot(y-y_0)\geq \Psi(y_0)+b(y_0)\cdot(y-y_0).
\ee
Hence $\Psi(y_0)=\Phi(y_0)$ and therefore $b(y_0)\in\partial\Psi(y_0)$. Hence by the result of Moreau $\Psi=\Phi+c$ for some constant $c$. But ${\rm dom}\,\partial\Phi$ is nonempty (for example because $\partial\Phi$ is maximal monotone) and so $c=0$ and $\Psi=\Phi$.
\end{remark}
\begin{remark}
\label{remark1} It does not seem obvious how to construct a smooth extension $\varphi(x,y)$ that is decreasing in $x$. This does not immediately follow  from the fact that $\tilde\varphi(x,y)$ is decreasing in $x$  because the mollification \eqref{6} averages $\tilde\varphi$ over a range of values of $y'$ that grows with $x$.
\end{remark}
\begin{remark}
\label{remark2}
 If $\Phi$ is not strictly convex on $\dom \Phi$ then the function $\varphi$ cannot in general be chosen to be strictly convex on $\R^n\times (0,\infty)$. Indeed if $\Phi=0$ then  $\varphi$ can only depend on $x$. To see this let $x>0$, $y,y'\in\R^n$ and for $\ep>0$ note that
\be
\label{15a} (x-\ep,y')=\frac{x-\ep}{x}(x,y)+\left(1-\frac{x-\ep}{x}\right)(0,z),
\ee
where $z:=\ep^{-1}(xy'-(x-\ep)y)$, so that by convexity
\be
\label{15b}
\varphi(x-\ep,y')&\leq &\frac{x-\ep}{x}\varphi(x,y)+ \left(1-\frac{x-\ep}{x}\right)\varphi(0,z)\nonumber\\
&=&\frac{x-\ep}{x}\varphi(x,y).
\ee
Letting $\ep\to 0$ we obtain $\varphi(x,y')\leq\varphi(x,y)$. Interchanging $y,y'$ we deduce that $\varphi(x,y)=\varphi(x,y')$ as required.
\end{remark}
\begin{remark}
\label{remark3}
An interesting open problem is to determine the pairs $\Phi_0$ and $\Phi_1$ of proper lower semicontinuous convex functions on $\R^n$  which are such that there is a convex function $\varphi:[0,1]\times \R^n\to(-\infty,\infty]$ that is finite on $(0,1)\times\R^n$ and  interpolates between $\Phi_0$ and $\Phi_1$ in the sense that $\varphi(0,y)=\Phi_0(y),\, \varphi(1,y)=\Phi_1(y)$ for all $y\in\R^n$ and
 \be
\label{15c}\lim_{x\to 0+}\varphi(x,y)=\Phi_0(y), \lim_{x\to 1-}\varphi(x,y)=\Phi_1(y) \text{ for each } y\in\R^n.
\ee
The set of such pairs $(\Phi_0,\Phi_1)$ is clearly convex. In the case $\Phi_0=0$, Remark \ref{remark2} shows that the only possibility is that $\Phi_1$ is constant, while in the case $\Phi_0=i_{\{0\}}$ Example \ref{ex1} below shows that any convex $\Phi_1:\R^n\to\R$ satisfying $\displaystyle\lim_{|y|\to\infty}\frac{\Phi_1(y)}{|y|}=\infty$ is possible.

Setting $C=\{0,1\}\times\R^n$ the problem is seen to be related to that of extending a convex function on $C\subset \R^s$ to a convex function on its convex hull ${\rm co}(C)$. This is studied for $C$ compact in \cite{bucicovschilebl2013} and for general $C$ in \cite{yan2014} (but without any assertion of continuity of the extension as $C$ is approached as in \eqref{15c}). When $C$ is compact and convex the question of extending a smooth convex function on $C$ to a smooth convex function on $\R^s$
is discussed in \cite{azagramudarra2019}.
\end{remark}
We give two examples of explicit constructions of convex extensions, using the two methods in the different proofs of Proposition \ref{finiteextension}. In neither example do we need to mollify $\tilde\varphi$ since it is already smooth.
\begin{example}
\label{ex1}
Let $\Phi=i_{\{0\}}$ be the indicator function of $0$ as described in the introduction. Then ${\rm dom}\,\partial\Phi=\{0\}$ and $\partial\Phi(0)=\R^n$, so that \eqref{formula} with $S$ given by $\{(0,b):b\in\R^n\}$ and $\psi(t)=c_p t^\frac{p}{p-1}$, $p>1$, $c_p=(p-1)p^\frac{-p}{p-1}$ gives
$\tilde\varphi(x,y)=\sup_{b\in\R^n}(b\cdot y-c_p|b|^\frac{p}{p-1}x)$. An elementary calculation then shows that
\be
\label{15}
\tilde\varphi(x,y)=\left\{\begin{array}{cl}\frac{|y|^p}{x^{p-1}},&x>0\\
i_{\{0\}}(y),&x=0 
\end{array}\right. ,
\ee
which is smooth for $x>0$ if $p$ is an even integer. In fact it is not hard to check that a more general convex extension which is smooth for $x>0$ is given by
\be
\label{15az}
\tilde\varphi(x,y)=\displaystyle\left\{\begin{array}{cl}x\eta(\frac{y}{x}),& x>0\\
i_{\{0\}}(y),&x=0
\end{array}\right.,
\ee
where $\eta:\R^n\to\R$ is convex, smooth, and such that $\displaystyle\lim_{|y|\to\infty}\frac{\eta(y)}{|y|}=\infty$.
\end{example}
\begin{example}
\label{log}
Let $n=1$ and 
\be
\label{logf}
\Phi(y)=\left\{\begin{array}{cl}-\ln y,&y>0\\\infty,&y\leq 0\end{array}\right..
\ee
Then an elementary calculation shows that
\begin{align}
\label{log1}
\tilde\varphi(x,y)&:=(\Phi\,\square\,\theta)\,(x,y)\nonumber\\
&=\left\{\begin{array}{cr}-\ln\left(\frac{1}{2}(y+\sqrt{y^2+2x})\right)+\frac{1}{4x}(\sqrt{y^2+2x}-y)^2,&x>0\\\Phi(y),&x=0\end{array}\right..
\end{align}
\end{example}

 \section{Polyconvexity conditions}
\label{polyconvexity}
In this section we give an application of Theorem \ref{lifting} to 3D nonlinear elasticity. Denote by $M^{3\times 3}$ the space of real $3\times 3$ matrices. Consider an elastic body occupying a bounded open set $\om\subset\R^3$ in a reference configuration. The total free energy at a constant temperature corresponding to a deformation $y:\om\to\R^3$ is given by 
\be
\label{16}
I(y)=\int_\om\psi(Dy(x))\,dx,
\ee
where  the free-energy density $\psi:M^{3\times 3}_+\to [0,\infty),$ and 
$M^{3\times 3}_+:= \{A\in M^{3\times 3}:  \det A>0\}$.

To help prevent interpenetration of matter it is usually assumed that 
\be
\label{16a}\psi(A)\to\infty\text{ as }\det A\to 0+,\ee
which implies that if $I(y)<\infty$ then 
$\det Dy(x)>0 \text{ for a.e. }x\in\om$.

In order to prove existence of an absolute minimizer of $I$ it is necessary to suppose, among other hypotheses, that $\psi$ satisfies  a suitable convexity condition. The convexity condition assumed in \cite{j8} (see \cite{ciarlet83} for a clear and more recent exposition) is that $\psi$ is {\it polyconvex}, that is
 there is a convex function $g:M^{3\times 3}\times M^{3\times 3}\times [0,\infty)\to \R\cup\{\infty\}$ such that 
\be 
\label{18}\psi(A)=g(A,\cof A, \det A) \text{ for all  }A\in M^{3\times 3}_+,
\ee
where $\cof A$ denotes the matrix of cofactors of $A$. Given $\delta>0$, define $E_\delta=\{(A,\cof A, \delta):\det A=\delta\}$. Since, as is proved in \cite[Theorem 4.3]{j8},   the convex hull of $E_\delta$ in $M^{3\times 3}\times M^{3\times 3}\times\R\cong\R^{19}$ is equal to $M^{3\times 3}\times M^{3\times 3}\times\{\delta\}$, it follows from \eqref{18} that $g(A,H,\delta)<\infty$ for all $A,H\in M^{3\times 3}$ and $\delta>0$.

In \cite{j8} it was further assumed that $g$ is continuous with $g(A,H,0)=\infty$ for all $A, H\in M^{3\times 3}$. Provided that $\psi(A)\to\infty$ as $|A|\to\infty$ this implies that \eqref{16a} holds.

Later, in \cite{j26} (see also \cite{ciarlet83}) it was observed that existence could be proved if one only assumes \eqref{16a} and that \eqref{18} holds for a convex $g:M^{3\times 3}\times M^{3\times 3}\times(0,\infty)\to\R$. But it is not immediately obvious that this really is a weaker hypothesis. Applying Theorem \ref{lifting} we see that it is.
\begin{theorem}
\label{pc}
There exists a smooth polyconvex function $\psi:M^{3\times 3}_+\to[0,\infty)$ satisfying \eqref{16a} for which the corresponding $g$ is continuous but does not satisfy $g(A,H,0)=\infty$ for all $A,H\in M^{3\times 3}$.
\end{theorem}
\begin{proof}
Let $V=\{(A,\cof A):\det A=0\}$.  Then $V$ is a closed subset of $M^{3\times 3}\times M^{3\times 3}$ and $(\1,\1)\not\in V$, where $\1$ is the identity $3\times 3$ matrix. Let $r>0$ be such that $\det A>0$ if $|A-\1|\leq r$, where $|\cdot|$ denotes the Euclidean norm on $M^{3\times 3}\cong\R^9$. Define $\Phi:M^{3\times 3}\times M^{3\times 3}\to [0,\infty]$ by
\be
\label{19}
\Phi(A,H)=\left\{\begin{array}{cc} \frac{1}{r^2-|A-\1|^2}+\frac{1}{r^2-|H-\1|^2}&\text{if }|A-\1|<r, |H-\1|<r\\ \infty&\text{otherwise}.\end{array}\right.
\ee
Then $\Phi$ is convex and continuous, so that by Theorem \ref{lifting} there exists a continuous convex  function $g:M^{3\times 3}\times M^{3\times 3}\times [0,\infty)\to [0,\infty]$ such that $g(A,H,0)=\Phi(A,H)$ for all $A,H\in M^{3\times 3}$ and $g(A,H,\delta)$ is smooth for $\delta>0$. Define $\psi(A)=g(A,\cof A,\det A)+|A|^2$. Then $\psi:M^{3\times 3}_+\to [0,\infty)$ is smooth and polyconvex, $\psi(A)\geq|A|^2$ and  $g(\1,\1,0)<\infty$. If $\det A^{(j)}\to 0+$ then we may assume either that $|A^{(j)}|\to \infty$, in which case $\psi(A^{(j)})\to\infty$, or  that $A^{(j)}\to A\in M^{3\times 3}$ with $\det A=0$, when 
$\psi(A^{(j)})=g(A^{(j)},\cof A^{(j)},\det A^{(j)})\to g(A,\cof A,0)=\infty$. Hence \eqref{16a} holds.\qed
\end{proof}

From the point of view of mechanics, Theorem \ref{pc} is unsatisfactory because the $\psi$ constructed does not satisfy the physically necessary 
 {\it frame-indifference condition}
\be
\label{20}
\psi(RA)=\psi(A)\text{ for all } R\in SO(3), A\in M^{3\times 3}_+,
\ee
which is not used for the proofs of existence in \cite{j8,j26}. In addition one would  like an example which is also {\it isotropic}, so that
\be
\label{iso}
\psi(AQ)=\psi(A) \text{ for all }Q\in SO(3), A\in M^{3\times 3}_+.
\ee
However we can adapt  Example \ref{ex1}  to give a frame-indifferent and isotropic example.\\

\begin{example}
\label{fiex5}
The frame-indifferent and isotropic function
\be
\label{fiex}
\psi(A)=\frac{|A|^2}{\det A}
\ee
is polyconvex with corresponding $g:M^{3\times 3}\times M^{3\times 3}\times[0,\infty)\to[0,\infty]$ given by
\be
\label{fiex1}
g(A,H,\delta)=\left\{\begin{array}{cl}\frac{|A|^2}{\delta},&A,H\in M^{3\times 3}, \delta>0\\0,&(A,H,\delta)=(0,0,0)\\
\infty,&\text{otherwise}
\end{array}\right.
\ee
and $\psi(A)\to\infty$ as $\det A\to 0+$.\\

\noindent That $g$ is convex and lower semicontinuous follows as for $\theta$ (see \eqref{thetaconv}), while Hadamard's inequality $|A|^3\geq 3^\frac{3}{2}\det A$ implies that $\psi(A)\geq 3(\det  A)^{-\frac{1}{3}}$. 
\end{example} 

If $\psi$ is polyconvex and frame-indifferent, we can without loss of generality suppose that the corresponding $g$ satisfies the invariance condition
\be
\label{invariance}
g(RA,RH,\delta)=g(A,H,\delta) \text{ for all }R\in SO(3), A,H\in M^{3\times 3}, \delta\in[0,\infty).
\ee
Indeed we can replace $g$ by
\be
\label{inv1}
\tilde g(A,H,\delta)=\av_{SO(3)} g(RA,RH,\delta)\,d\mu(R),
\ee
where $$\av_{SO(3)}f(R)\,d\mu(R):=\frac{\int_{SO(3)}f(R)\,d\mu(R)}{\mu(SO(3))}$$ and $\mu$ denotes Haar measure on $SO(3)$. Then $\tilde g$  satisfies \eqref{invariance}, is convex, and by \eqref{20} and the relation $\cof (RA)=R\,\cof A$ we have 
\begin{align}
\label{inv2}
\psi(A)&=g(A,\cof A,\det A)\nonumber\\&=\av_{SO(3)}g(RA,R\,\cof A,\det A)\,d\mu(R)\nonumber\\&=\tilde g(A,\cof A,\det A).
\end{align}
But, as is well known, $0$ belongs to the convex hull of $SO(3)$. Explicitly, $0= \frac{1}{4}\sum_{i=0}^3R_i$, where $R_0=\1, R_i=-\1+e_i\otimes e_i$ for $i=1,2,3$, and $e_i$ is the unit vector in the $i^{\rm th}$ coordinate direction. So for any $A,H$
\be
\tilde g(0,0,0)&=&\tilde g\left(\sum_{i=0}^3\frac{1}{4}R_iA,\sum_{i=0}^3\frac{1}{4}R_iH,0\right)\nonumber\\
&\leq&\sum_{i=0}^3\frac{1}{4}\tilde g(R_iA,R_iH,0)\nonumber\\
&=&\left(\sum_{i=0}^3\frac{1}{4}\right)\tilde g(A,H,0)=\tilde  g(A,H,0),\label{22}
\ee
so that $\tilde g(0,0,0)=\infty$ implies $\tilde g(A,H,0)=\infty$ for all $A,H$. Thus to construct an example we need $\tilde g(0,0,0)<\infty$, as in Example \ref{fiex5}.

In Theorem \ref{pc} $g(\cdot,\cdot,0)$ is finite on an open subset of $M^{3\times 3}\times M^{3\times 3}$. However no such example is possible for $\tilde g$. Indeed, if $\tilde g(\cdot,\cdot,0)<\infty$  on an open set $U\subset M^{3\times 3}\times M^{3\times 3}$ then $\tilde g(A,H,0)<\infty$ for $(A,H)$ in the open set $\{(RA,RH):(A,H)\in U, R\in SO(3)\}$, the convex hull of which is therefore open, and which contains $(0,0)$ by \eqref{22}. Therefore $g(A,H,0)<\infty$ for $(A,H)$ in some open ball $B(0,r)$ with centre $0\in M^{3\times 3}\times M^{3\times 3}$. Since $g(\cdot,\cdot,0)$ is convex, it is continuous and bounded on $B(0,r/2)$. Similarly $\tilde g(\cdot,\cdot, 1)$ is bounded on $B(0,r/2)$, so that by convexity $\tilde g$ is bounded on $B(0,r/2)\times [0,1]$. But then $\psi(A)=\tilde g(A,\cof A,\det A)$ is bounded as $A\to 0$ with $\det A>0$.
\section*{Acknowledgement}\noindent
 JMB was supported by EPSRC through grant  EP/V00204X and is grateful to Gilles Francfort, Jan Kristensen and Mark Peletier for their interest and discussions. CLH was supported by EPSRC through grant EP/L016508/1.


\section*{Declarations}
The authors have no conflict of interest.
\section*{Data availability statement}
There is no data associated with this work.


\begin{thebibliography}{10}

\bibitem{azagramudarra2019}
D.~Azagra and C.~Mudarra.
\newblock Smooth convex extensions of convex functions.
\newblock {\em Calc. Var. Partial Differential Equations}, 58(3):Paper No. 84,
  27, 2019.

\bibitem{j8}
J.~M. Ball.
\newblock Convexity conditions and existence theorems in nonlinear elasticity.
\newblock {\em Arch. Ration. Mech. Anal.}, 63:337--403, 1977.

\bibitem{j26}
J.~M. Ball and F.~Murat.
\newblock {$W^{1,p}$}-quasiconvexity and variational problems for multiple
  integrals.
\newblock {\em J. Functional Analysis}, 58:225--253, 1984.

\bibitem{bauschkeetal}
H.~H. Bauschke and P.~L. Combettes.
\newblock {\em Convex analysis and monotone operator theory in {H}ilbert
  spaces}.
\newblock CMS Books in Mathematics/Ouvrages de Math\'{e}matiques de la SMC.
  Springer, New York, 2011.
\newblock With a foreword by H\'{e}dy Attouch.

\bibitem{benamoubrenier}
J.-D. Benamou and Y.~Brenier.
\newblock A computational fluid mechanics solution to the {M}onge-{K}antorovich
  mass transfer problem.
\newblock {\em Numerische Mathematik}, 84(3):375--393, 2000-1-1.

\bibitem{bucicovschilebl2013}
O.~Bucicovschi and J.~Lebl.
\newblock On the continuity and regularity of convex extensions.
\newblock {\em J. Convex Anal.}, 20(4):1113--1126, 2013.

\bibitem{ciarlet83}
P.~G. Ciarlet.
\newblock {\em Mathematical Elasticity, Vol.I: Three-Dimensional Elasticity}.
\newblock North-Holland, 1988.

\bibitem{ekelandtemam}
I.~Ekeland and R.~T\'emam.
\newblock {\em Convex analysis and variational problems}, volume~28 of {\em
  Classics in Applied Mathematics}.
\newblock Society for Industrial and Applied Mathematics (SIAM), Philadelphia,
  PA, {E}nglish edition, 1999.
\newblock Translated from the French.

\bibitem{moreau65}
J.-J. Moreau.
\newblock Proximit\'{e} et dualit\'{e} dans un espace hilbertien.
\newblock {\em Bull. Soc. Math. France}, 93:273--299, 1965.

\bibitem{moreau67}
J.-J. Moreau.
\newblock Fonctionnelles convexes.
\newblock {\em S\'{e}minaire Jean Leray}, (2):1--108, 1966-1967.

\bibitem{phelps}
R.~R. Phelps.
\newblock {\em Convex functions, monotone operators and differentiability},
  volume 1364 of {\em Lecture Notes in Mathematics}.
\newblock Springer-Verlag, Berlin, second edition, 1993.

\bibitem{rockafellar70}
R.~T. Rockafellar.
\newblock {\em Convex analysis}.
\newblock Princeton University Press, Princeton, New Jersey, 1970.

\bibitem{yan2014}
M.~Yan.
\newblock Extension of convex function.
\newblock {\em J. Convex Anal.}, 21(4):965--987, 2014.

\end{thebibliography}
\end{document}